\documentclass{article}

\usepackage{amsfonts}
\usepackage{amsmath}
\usepackage{a4wide}
\usepackage{amssymb}
\usepackage{url}

\newtheorem{lem}{Lemma}[section]
\newtheorem{theo}[lem]{Theorem}
\newtheorem{coro}[lem]{Corollary}
\newtheorem{propo}[lem]{Proposition}
\newtheorem{rema}[lem]{Remark}
\newtheorem{defi}[lem]{Definition}

\newenvironment{lemma}
{\begin{lem}\sl } {\end{lem}}

\newenvironment{theorem}
{\begin{theo}\sl } {\end{theo}}

\newenvironment{corollary}
{\begin{coro}\sl } {\end{coro}}

\newenvironment{proposition}
{\begin{propo}\sl } {\end{propo}}

\newenvironment{remark}
{\begin{rema}\rm } {\end{rema}}

\newenvironment{proof}{\paragraph*{Proof}}
{\par}

\newcommand\qed{\hfill$\square$}

\newcommand\CC{{\mathcal C}}
\newcommand\OO{{\mathcal O}}

\newcommand\calF{{\mathcal F}}
\newcommand\calI{{\mathcal I}}

\newcommand\gal{{\mathrm{Gal}}}
\newcommand\ord{{\mathrm{ord}}}
\newcommand\GL{{\mathrm{GL}}}
\newcommand\SL{{\mathrm{SL}}}

\newcommand\eps\varepsilon
\newcommand\ph\varphi
\newcommand\C{{\mathbb C}}
\newcommand\R{{\mathbb R}}
\newcommand\F{{\mathbb F}}
\newcommand\Q{{\mathbb Q}}
\newcommand\PPP{{\mathbb P}}
\newcommand\Z{{\mathbb Z}}
\newcommand\height{{\mathrm h}}
\newcommand\im{{\mathrm {Im}\,}}
\newcommand\sign{{\mathrm {sign}\,}}
\newcommand\HH{{\mathcal H}}
\newcommand\bfa{{\mathbf a}}
\newcommand\bfb{{\mathbf b}}

\newcommand\bfA{{\mathbf A}}

\newcommand\tilD{{\widetilde D}}

\newcommand\topbot[2]{{\genfrac{}{}{0pt}{}{{#1}}{{#2}}}}

\newcommand\tors{{\mathrm{tors}}}
\newcommand\cl{{\mathrm{cl}}}

\newcommand\tilJ{{\widetilde J}}  
\newcommand\tilg{{\tilde g}}

\newcommand\splic{{\mathrm{sp.C}}}

\title{Runge's Method and Modular Curves}
 
\author{Yuri Bilu, Pierre Parent (Universit\'e de Bordeaux~I)}

1

\begin{document}

\maketitle

\begin{abstract}
We bound the $j$-invariant of $S$-integral points on arbitrary modular curves over arbitrary
fields, in terms of the congruence group defining the curve, assuming a certain {\it Runge 
condition} is satisfied by our objects. We then apply our bounds to prove that for sufficiently 
large prime $p$, the points of $X_0^+ (p^r )(\Q )$ with $r>1$ are either cusps or CM points. 
This can be interpreted as the non-existence of quadratic elliptic $\Q$-curves with higher 
prime-power degree.
\medskip

AMS 2000 Mathematics Subject Classification  11G18 (primary), 11G05, 11G16 
(secondary). 
\end{abstract}

{\footnotesize

\tableofcontents

}

\section{Introduction}
\addtocontents{toc}{\vspace{-0.7\baselineskip}}

Let ${N\ge 2}$ be an integer and $X(N)$ the principal modular curve of level~$N$. Further, 
let~$G$ a subgroup of $\GL_{2} (\Z /N\Z )$ containing $-1$  and let $X_G$ be the corresponding modular 
curve. This curve is defined over $\Q  (\zeta_{N})^{\det (G)}$, so in particular it is 
defined over $\Q$ if ${\det (G)=(\Z /N\Z )^\times}$. (Through all this paper, we say that 
an algebraic curve is  \textsl{defined} over a field if it has a geometrically integral 
model 
over this basis.) As usual, we denote by~$Y_G$ the finite part of~$X_G$ (that is,~$X_G$ 
deprived  of the cusps). If~$X_G$ is defined over a number field~$K$, the curve~$X_{G}$
has a natural (modular) model over ${\OO=\OO_K}$ that we still denote by~$X_{G}$. The 
cusps define a closed subscheme of~$X_{G}$ over~$\OO$, and we also call~$Y_G$ the 
relative curve over~$\OO$ which is~$X_{G}$ deprived of the cusps. If~$S$ is a finite set of 
places of~$K$ containing the infinite places, the set of $S$-integral points $Y_G
(\OO_S)$ consists of those ${P\in Y_G(K)}$ for which ${j(P)\in \OO_S}$, where~$j$ is, as 
usual, the modular invariant and ${\OO_S=\OO_{K,S}}$ is the ring of~$S$-integers. 

In its simplest form, the first principal result of this article gives an explicit upper bound for
${j(P)\in \Z}$ under certain Galois condition on the cusps. 

\begin{theorem}
\label{th1}
Assume that~$X_G$ is defined over~$\Q$, and assume that the absolute Galois group 
$\gal(\bar\Q/\Q)$ acts non-transitively on the cusps of~$X_G$. Then for any ${P\in Y_G
(\Z)}$ we have
\begin{equation}
\label{eth1}
\log|j(P)|\le 12|G|N^2\log 3N. 
\end{equation}
\end{theorem}

This result was announced in~\cite{BP08}. Because of an inaccuracy in the proof given 
in~\cite{BP08}, the $\log$-factor is missing therein (this, however, does not affect 
the arithmetical applications of that theorem, which by the way has since been drastically 
improved in \cite{BP09}, see below). 

Actually, we obtain a more general Theorem~\ref{tbo} below,  which applies to any 
number field and any ring of~$S$-integers in it. To state Theorem~\ref{tbo} we need to 
introduce some notations. We denote by 
${\height(\cdot)}$ the usual absolute logarithmic height (see Subsection~\ref{ssnota}).  
For ${P\in X_G(\bar \Q)}$ we shall write ${\height(P)=\height\bigl(j(P)\bigr)}$.
For a number field~$K$ we denote by ${\CC=\CC(G)}$ the set of cusps of~$X_G$, and by 
$\CC(G,K)$ the set of $\gal(\bar K/K)$-orbits of~$\CC$.

\begin{theorem}
\label{tbo}
Let~$K$ be a number field and~$S$ a finite set of places of~$K$ (containing all the 
infinite places). Let~$G$ be a subgroup of $\GL_2(\Z/N\Z)$ such 
that~$X_G$ is defined over~$K$.  Assume that 
${|\CC(G,K)|> |S|}$
(the ``Runge condition''). Then for any ${P\in Y_G(\OO_S)}$ we have 
\begin{equation}
\label{etbo}
\height(P) \le 36 s^{s/2+1}\left(N^2|G|/2\right)^s\log 2N, 
\end{equation}
where ${s=|S|}$. If ${S=M_K^\infty}$, we even have
\begin{equation}
\label{etbo1}
\height(P) \le 24 s^{s/2+1}\left(N^2|G|/2\right)^s\log 3N.
\end{equation}
\end{theorem}

Theorem~\ref{th1} is a special case of this theorem.  

This theorem is proved  in Section~\ref{sproof}  by a variation of the 
method of Runge, after some preparation in Sections~\ref{sest},~\ref{snecu} 
and~\ref{smuni}.  For a general discussion of Runge's method see~\cite{Bo83,Le08}. 
 
Theorems~\ref{th1} and~\ref{tbo} apply for a general group~$G$. For a specific~$G$, 
one may obtain much stronger results. One such example can be found in~\cite{BP09}. In 
Section~\ref{sspc} of the present article we apply our general strategy in the case when 
${N=p}$ is a prime number, the field~$K$ is quadratic and~$G$ is a split Cartan subgroup; 
in this case we obtain a much sharper estimate for the integral points, than what follows by 
direct application of Theorem~\ref{tbo}. In Section~\ref{sqcurves} we apply this result to 
the arithmetic of modular curves, using the integrality property established in 
Section~\ref{sapp}. 

\begin{theorem}
\label{txpr}
There is an absolute constant $p_{0}$ such that, if ${r>1}$ and ${p>p_{0}}$, the set $X_{0}^+ (p^r )(\Q )$ is 
made of cusps and CM points. Equivalently, there is no quadratic $\Q$-curve of prime power 
degree $p^r$ with exponent $r>1$ when this prime is large enough.       
\end{theorem}

Recall that a \textsl{$\Q$-curve} is an elliptic curve without complex multiplication over 
$\bar\Q$ which is isogenous to each of its conjugates over~$\Q$. When such a curve is 
quadratic (that is, defined over a quadratic field), we say it \textsl{has degree~$N$} if the 
minimal degree of an isogeny from the curve to its only non-trivial conjugate is~$N$. For 
concrete examples of quadratic $\Q$-curves see~\cite{Ga02} and references therein. 

Theorem~\ref{txpr} extends the principal result of~\cite{BP09} and 
essentially settles the problem tackled in~\cite{Mo86}, \cite{Pa05} or 
\cite{Re08}.

\subsection{Notations, Conventions}
\label{ssnota}
Everywhere in this article  $\log$ and~$\arg$ stand for the principal branches of the 
complex logarithm and argument functions; that is, for any ${z\in \C^\times}$ we have
${-\pi<\im\log z=\arg z\le \pi}$. In Section 2 of this article we shall systematically use, 
often without special reference, the estimates of the kind
\begin{equation}
\label{eschwa}
\begin{aligned}
|\log (1+z)| &\le \frac{|\log(1-r)|}r|z|, \qquad
\left|e^z-1\right|\le \frac{e^r-1}r|z|, \\
\left|(1+z)^A-1-Az\right|&\le \frac{\left|(1+\eps r)^A-1-\eps Ar
\right|}{r^2}|z|^2 \quad (A\in \R ,\ \eps=\sign A), 
\end{aligned} 
\end{equation}
etc., for ${|z|\le r<1}$. They can be easily deduced from the maximum principle.

Let~$\HH$ denote Poincar\'e upper half-plane: ${\HH=\{\tau\in \C : \im\tau>0\}}$. For 
${\tau\in \HH}$ we put ${q_\tau=e^{2\pi i \tau}}$ (or simply $q$ if there is no 
ambiguity). We put ${\bar\HH=\HH\cup\Q\cup \{i\infty\}}$. If~$\Gamma$ is the 
pull-back of ${G\cap \SL_2(\Z/N\Z)}$ to $\SL_2(\Z)$, then the set $X_G(\C)$ of complex 
points is analytically isomorphic to the quotient  ${X_\Gamma=\bar\HH/\Gamma}$, supplied 
with the properly defined topology and analytic structure \cite{La76,Sh71}.  

We denote by~$D$ the standard fundamental 
domain of $\SL_2(\Z)$ (the hyperbolic triangle with vertices $e^{\pi i/3}$, 
$e^{2\pi i/3}$ and $i\infty$, together with the geodesic segments ${[i,e^{2\pi i/3}]}$ 
and ${[e^{2\pi i/3},i\infty]}$).
Notice that for ${\tau\in D}$ we have ${|q_\tau|\le e^{-\pi\sqrt3}<0.005}$, 
which will be systematically used without special reference.

For ${\bfa=(a_1,a_2)\in \Q^2}$ we put ${\ell_\bfa=B_2\bigl(a_1-\lfloor a_1
\rfloor\bigr)/2}$ where ${B_2(T)=T^2-T+1/6}$ is the second Bernoulli polynomial. The 
quantity~$\ell_\bfa$ is $\Z^2$-periodic in~$\bfa$ and is thereby well-defined for 
${\bfa\in (\Q/\Z)^2}$ as well: for such~$\bfa$ we have ${\ell_\bfa=B_2(
\widetilde{a}_1)}$, where~$\widetilde{a}_1$ is the lifting of the first coordinate 
of~$\bfa$ to  the interval $[0,1)$. Obviously, ${|\ell_\bfa|\le 1/12}$;
this will  also be often used without special reference. 

We fix, once and for all, an algebraic closure~$\bar\Q$ of~$\Q$, which is assumed to be a 
subfield of~$\C$. In particular, for every ${a\in\Q}$ we have the well defined root of unity 
${e^{2\pi i a}\in \bar\Q}$.
Every number field used in this article is presumed to be  contained in the fixed~$\bar\Q$.
If~$K$ is such a number field and~$v$ is a valuation on~$K$, then we tacitly assume
than~$v$ is somehow extended to ${\bar\Q=\bar K}$; equivalently, we fix an algebraic
closure~$\bar K_v$ and an embedding ${\bar\Q\hookrightarrow\bar K_v}$. In
particular, the roots of unity $e^{2\pi i a}$ are well-defined elements of~$\bar K_v$. 

For a number field~$K$ we denote by~$M_K$ the set of all valuations (or places) of~$K$ 
normalized to extend the usual infinite and $p$-adic valuations of~$\Q$: ${|2|_v=2}$ if ${v
\in M_K}$ is infinite, and ${|p|_v=p^{-1}}$ if~$v$ extends the $p$-adic valuation of~$\Q$. 
In the finite case we sometimes use the additive notation $v(\cdot)$, normalized to have 
${v(p)=1}$. We denote by~$M_K^\infty$ and~$M_K^0$ the subsets of~$M_K$ consisting 
of the infinite (archimedean) and the finite (non-archimedean) valuations, respectively. 

Recall the definition of the absolute logarithmic height $\height(\cdot)$. For 
${\alpha \in \bar\Q}$ we pick a number field~$K$ containing~$\alpha$ and put 
${\height(\alpha) = [K:\Q]^{-1}\sum_{v\in M_K}[K_v:\Q_v]\log^+|\alpha|_v}$,
where the valuations on~$K$ are normalized to extend standard infinite and $p$-adic 
valuations on~$\Q$ and ${\log^+x=\log\max\{x,1\}}$.  The value of $\height(\alpha)$
is known to be independent on the particular choice of~$K$. As usual, we extend the 
definition of the height to ${\PPP^1(\bar \Q)=\bar\Q\cup\{\infty\}}$ by setting 
${\height (\infty)=0}$. If~$\alpha$ is a rational integer or an imaginary quadratic 
integer then ${\height (\alpha)=\log|\alpha|}$.  

  If $E$ is an elliptic over $\overline{\Q}$, define $\height (E):=\height (j_{E})$ for 
$j_{E}$ its $j$-invariant. For $P\in X(\bar\Q )$ a point with values in $\bar\Q$
of a modular curve $X$, the height we will use (unless explicitely mentioned otherwise) will 
be this naive Weil height of the elliptic curve associated to~$P$, that is $\height (P)=
\height (j(P)).$

\section{Estimates for  Modular Functions at Infinity}
\addtocontents{toc}{\vspace{-0.7\baselineskip}}
\label{sest}

The results of this section must be known, but we did not find them in the available 
literature, so we state and prove them here. For the sake of further applications, we have tried
to give rather sharp constants, even if we do not actually need this precision in the sequel of 
the present paper.  

\subsection{Estimating the $j$-Function}

Recall that the modular 
$j$-invariant ${j:\HH\to\C}$ is defined by ${j(\tau)=(12c_2(\tau))^3/\Delta(\tau)}$, 
where 
$$
c_2(\tau)= \frac{(2\pi )^4}{12}\left(1+240\sum_{n=1}^\infty
\frac{n^3q^n}{1-q^n}\right)
$$
(see, for instance, \cite[Section~4.2]{La73}) and
${\Delta(\tau) =(2\pi )^{12}q\prod_{n=1}^\infty(1-q^n)^{24}}$. Also,~$j$ has the 
familiar Fourier expansion
${j(\tau)= q^{-1}+ 744+196884q+\ldots}$.

\begin{proposition}
\label{pqj}
For ${\tau\in \HH}$ such that ${|q|(=|q_\tau|)\le 0.005}$ (and, in particular, for every
${\tau\in D}$) we have 
\begin{equation}
\label{ejq}
\left|j(\tau)-q^{-1}-744\right| \le 330000|q |.
\end{equation}
\end{proposition}
(Recall that~$D$ is the standard fundamental domain for $\SL_2(\Z)$.)

\begin{proof}
Using the estimate ${n^3\le 3^n}$ for ${n\ge 3}$, we find that 
for ${|q|<1/3}$ 
\begin{align*}
\left|\frac{12}{(2\pi )^4}c_2(\tau)-1-240q\right| &\le 240\left(
\frac{|q|^2}{1-|q|}+ \sum_{n=2}^\infty\frac{n^3|q|^n}{1-|q|^n}\right)\\
&\le \frac{240}{1-|q|}\left(|q|^2+8|q|^2+\sum_{n=3}^\infty|3q|^n\right)\\
&=\frac{2160}{(1-|q|)(1-3|q|)}|q|^2,
\end{align*}
and for ${|q|\le 0.005}$ we obtain 
\begin{equation}
\label{ectwo}
\left|\frac{12}{(2\pi )^4}c_2(\tau)-1-240q\right|\le 2204|q|^2.
\end{equation}
Further, using~(\ref{eschwa}), we obtain,  for ${|q|\le0.005}$, 
$$
\left|\log\frac{(2\pi )^{12}q(1-q)^{24}}{\Delta(\tau)}\right|=
24\left|\sum_{n=2}^\infty\log\left(1-q^n\right)\right| \le 24.1
\sum_{n=2}^\infty |q|^n \le 24.3|q|^2.
$$
Hence
\begin{align*}
\left|\frac{(2\pi )^{12}q}{\Delta(\tau)}-1-24q\right|&\le 
\left|(1-q)^{-24}\right|\left|\frac{(2\pi )^{12}q(1-q)^{24}}{\Delta(\tau)}-1
\right|+\left|(1-q)^{-24}-1-24q\right|\\
&\le 1.13\left|\log\frac{(2\pi )^{12}q(1-q)^{24}}{\Delta(\tau)}\right|+314|q|^2
\le 342|q|^2.
\end{align*}
Combining this with~(\ref{ectwo}), we obtain~(\ref{ejq}) after a tiresome, but 
straightforward calculation.\qed
\end{proof}

\bigskip

The following consequences are obtained by direct numerical computations.

\begin{corollary}
\label{cdplus}
For ${\tau\in D}$ the following statements are true. 
\begin{enumerate}
\item
\label{i2200}
We have ${\bigl|\log |q_\tau|\bigr|\le \log\bigl(|j(\tau)|+2400\bigr)}$.

\item
\label{i2500}
We have either ${|j(\tau)|\le 3500}$ or ${|q_\tau|< 0.001}$. 

\item
\label{iverysimple}
If ${|j(\tau)|> 3500}$ then ${\left|j(\tau)-q_\tau^{-1}\right|\le 1100}$ and
${\frac32 |j(\tau)|\ge \left|q_\tau^{-1}\right|\ge \frac12 |j(\tau)|}$. \qed
\end{enumerate}
\end{corollary}

\subsection{Estimating Siegel's Functions}
\label{sssieg}
For a rational number~$a$ we define ${q^a=e^{2\pi i a\tau}}$. 
Let ${\bfa=(a_1,a_2)\in \Q^2}$ be such that ${\bfa\notin \Z^2}$, and let ${g_\bfa:\HH
\to \C}$ be the corresponding \emph{Siegel function} \cite[Section~2.1]{KL81}. Then 
we have the following infinite product presentation for~$g_\bfa$ 
\cite[page~29]{KL81} (where $B_2(T)$ is the second Bernoulli polynomial):
\begin{equation}
\label{epga}
g_\bfa(\tau)= -q^{B_2(a_1)/2}e^{\pi ia_2(a_1-1)}\prod_{n=0}^\infty\left(
1-q^{n+a_1}e^{2\pi ia_2}\right)\left(1-q^{n+1-a_1}e^{-2\pi i a_2}\right) .
\end{equation}
We also have \cite[pages 29--30]{KL81} the relations
\begin{align}
\label{eperga}
g_\bfa\circ\gamma &=g_{\bfa\gamma} \cdot(\text{a root of unity}) \quad \text{for}
\quad \gamma\in\Gamma(1),\\
\label{eaa'}
g_\bfa&=g_{\bfa'}\cdot(\text{a root of unity})
 \quad \text{when} \quad \bfa\equiv\bfa'\mod\Z^2.
\end{align}
Remark that the root of unity in~(\ref{eperga}) is of order dividing~$12$, and 
in~(\ref{eaa'}) of order dividing $2N$, where~$N$ is the denominator of~$\bfa$ (the 
common denominator of~$a_1$ and~$a_2$; see~\cite{BP09} for more references).

The order of vanishing of~$g_\bfa$ at $i\infty$ (that is, the only rational number~$\ell$ 
such that the limit ${{\displaystyle\lim_{\tau\to i\infty}}q_\tau^{-\ell}g_\bfa(
\tau)}$ exists and is non-zero) is equal to the number~$\ell_\bfa$, defined in 
Subsection~\ref{ssnota} (see \cite[page~31]{KL81}). 

\begin{proposition}
\label{psiar}
Let~$\bfa$ be an element of ${\Q^2\smallsetminus\Z^2}$ and ${N>1}$ an integer such 
that ${N\bfa \in \Z^2}$. Then  for ${\tau \in D}$  we have
\begin{equation}
\label{elogn}
\Bigl|\log \left|g_\bfa(\tau)\right|- \ell_\bfa\log |q_\tau|\Bigr| \le \log N.
\end{equation}
\end{proposition}

\begin{proof}
Due to~(\ref{eaa'}), we may assume that ${0\le a_1<1}$ and distinguish between the 
cases ${0<a_1<1}$ and ${a_1=0}$. According to~(\ref{epga}), the left-hand side 
of~(\ref{elogn}) is equal to 
$$
\left|\log\left|1-q^{a_1}\zeta\right|+\log\left|1-q^{1-a_1}\bar\zeta\right|+
\sum_{n=1}^\infty\left(\log \left|1-q^{n+a_1}\zeta\right| +\log\left|1-q^{n+1-
a_1}\bar\zeta\right|\right)\right|, 
$$
where ${\zeta=e^{2\pi i a_2}}$. Recall that, since ${\tau \in D}$, we have 
${|q|\le e^{-\pi\sqrt3}}$. Hence, for ${n\ge 1}$ each of $\left|q^{n+a_1}\right|$ and 
$\left|q^{n+1-a_1}\right|$  does not exceed $0.005$, whence
\begin{equation}
\label{esumn}
\left|\sum_{n=1}^\infty\left(\log \left|1-q^{n+a_1}\zeta\right| +\log\left|1-
q^{n+1-a_1}\bar\zeta\right|\right)\right|\le 1.005\frac{\left|q^{1+a_1}\right|
+\left|q^{2-a_1}\right|}{1-|q|}\le 3|q| \le 0.02.
\end{equation}
To estimate
\begin{equation}
\label{eloglog}
\bigl|\log|1-q^{a_1}\zeta|+\log|1-q^{1-a_1}\bar\zeta|\bigr|, 
\end{equation}
assume first that ${0<a_1<1}$. Then among the numbers $|q^{a_1}|$ and $|q^{1-a_1}|$ one 
is bounded by $|q|^{1/N}$ and the other is bounded by $|q|^{1/2}$, which 
bounds~(\ref{eloglog}) by
$$
\left|\log\bigl|1-e^{-\pi\sqrt3/N}
\bigr|\right|+ \left|\log\bigl|1-e^{-\pi\sqrt3/2}\bigr|\right|
\le \log N -0.1.
$$
This proves~(\ref{elogn}) in the case ${0<a_1<1}$. When ${a_1=0}$ then ${a_2
\notin\Z}$ and ${\zeta \ne 1}$, which bounds~(\ref{eloglog}) by 
$$
\left|\log \bigl |1-e^{2\pi i/N}\bigr|\right| + \left|\log\bigl|1-e^{-\pi\sqrt3}
\bigr|\right|
\le \log N -0.1, 
$$
proving~(\ref{elogn}) in this case as well. \qed
\end{proof}

\bigskip

Since Siegel's functions have no pole nor zero on the upper half plane~$\HH$, they should be 
bounded from above and from below on any compact subset of~$\HH$. In particular, they 
should be bounded where~$j$ is bounded. Here is a quantitative version of this. 

\begin{corollary}
\label{cgalar}
Let ${\bfa \in \Q^2\smallsetminus\Z^2}$ and ${N>1}$ satisfy ${N\bfa \in \Z^2}$. 
Then for any ${\tau\in 
\HH}$ we have     
\begin{equation}
\label{esmallj}
\bigl|\log |g_\bfa(\tau)|\bigr|\le \frac1{12}\log\bigl(|j(\tau)|+2400\bigr)+\log 
N.
\end{equation}
\end{corollary}

\begin{proof}
Replacing~$\tau$ by $\gamma\tau$ and~$g_\bfa$ by~$g_{\bfa\gamma^{-1}}$ with a 
suitable ${\gamma\in \Gamma(1)}$, we may assume that ${\tau\in D}$,  in 
particular~(\ref{elogn}) holds. Combining~(\ref{elogn}) with item~(\ref{i2200}) of Corollary~\ref{cdplus},  and using the inequality ${|\ell_\bfa|\le 1/12}$,  we obtain~(\ref{esmallj}). \qed 
\end{proof}

\subsection{Non-Archimedean Versions}

We also need non-archimedean versions of some of the above inequalities. In this 
subsection~$K_v$ is a field complete with respect to a non-archimedean valuation~$v$ 
and~$\bar K_v$ its algebraic closure. Let ${q\in K_v}$ satisfy ${|q|_v<1}$. For  
${\bfa=(a_1,a_2)\in \Q^2}$ such that ${\bfa\notin\Z^2}$    define~$g_\bfa(\tau)$ as
in~(\ref{epga}). Recall (see Subsection~\ref{ssnota}) that the roots of unity $e^{2\pi ia}$
with ${a\in \Q}$ are defined as elements of $\bar K_v$.

The right-hand side of~(\ref{epga}) is, formally, not well-defined because we use rational 
powers of~$q$. However, if we 
fix ${q^{1/12N^2}\in\bar K_v}$, where~$N$ is the order of~$\bfa$ in $(\Q/\Z)^2$, then 
everything becomes well-defined, and we again have~(\ref{eperga}) 
and~(\ref{eaa'}). The statement of the following proposition is independent on the 
particular choice of~$q^{1/12N^2}$.   

\begin{proposition}
\label{psinar}
Let~$\bfa$ be an element of ${\Q^2\smallsetminus\Z^2}$ and ${N>1}$ an integer such 
that ${N\bfa \in \Z^2}$. Then 
$$
\bigl|\log\left|g_\bfa(q)\right|_v-\ell_\bfa\log|q|_v\bigr|
\begin{cases}=0&\text {if ${v(N)=0}$}\\
\le \frac{\log p}{p-1}& \text{if ${v\mid p\mid N}$},
\end{cases}
$$
where here and below ${v\mid p\mid N}$ means that ${v(N)>0}$ and~$p$ is the prime below~$v$. 
\end{proposition}

\begin{proof}
We again may assume that ${0\le a_1<1}$. 
When 
${0< a_1<1}$ we have
${\log\left|g_\bfa(q)\right|_v=\ell_\bfa\log|q|_v}$. 
When ${a_1=0}$ we have 
$$
\log\left|g_\bfa(q)\right|_v=\ell_\bfa\log|q|_v + \log|1-e^{2\pi ia_2}|_v,
$$ 
and ${|1-e^{2\pi ia_2}|_v=1}$ if ${v(N)=0}$, while if ${v\mid p\mid N}$ we have ${1
\ge |1-e^{2\pi ia_2}|_v\ge p^{-1/(p-1)}}$.
\qed
\end{proof}

\section{Locating a ``Nearby Cusp''}
\addtocontents{toc}{\vspace{-0.7\baselineskip}}
\label{snecu}
Let~$N$ be a positive integer,~$G$ a subgroup of ${\GL_2(\Z/N\Z)}$ and~$X_G$ the 
corresponding modular curve, defined over a number field~$K$. We fix a valuation~$v$ 
of~$K$. Recall we denote by~$\OO$ the ring of integers of~$K$ and by~$K_v$ the 
$v$-completion of~$K$. When~$v$ is non-archimedean, we denote by~$\OO_{v}$ the ring of
integers of~$K_{v}$, and by~$k_{v}$ its residue field at~$v$. As usual~$\zeta_{N}$ will 
denote a primitive~$N$-th root of unity. 
Recall that when we say that a curve ``is defined over'' a field, it means that this curve has a 
geometrically integral model over that field.

Let~$P$ be a point on $X_G(K_v)$ such that $|j(P)|$ is ``large''. Then it is intuitively clear 
that, in the $v$-adic metric,~$P$ is situated ``near'' a cusp of~$X_G$. The purpose of this 
section is to make this intuitive observation precise and explicit. We shall locate this 
``nearby'' cusp and specify what the word ``near'' means.

We first recall the following description of the cuspidal locus of~$X(N)$ (for more details 
see e.g.~\cite[Chapitres V and VII]{DR73}). The cusps of~$X(N)$ define a closed 
subscheme of the smooth locus of the modular model of~$X(N)$ over~$\Z [\zeta_{N}]$. 
Fix a uniformization~$X(N)(\C )\simeq\bar{\HH} /\Gamma (N)$, let~$c_\infty$ be the 
cusp corresponding to ${i\infty\in\bar{\HH}}$, and write ${q^{1/N} =e^{2i\pi\tau/
N}}$ the usual parameter. If ${c=\gamma (c_\infty )}$, for some ${\gamma \in
\SL_{2} (\Z )}$, is another cusp, denote by ${q_c :=q\circ \gamma^{-1}}$ the parameter 
on~$X(N)(\C )$ at $c$. It follows from~\cite[Chapitre VII, Corollaire 2.5]{DR73} that the 
completion of the curve $X(N)$ over $\Z [\zeta_{N}]$ along the section $c$ is isomorphic 
to ${{\mathrm{Spec}}\bigl(\Z [\zeta_{N}]\,
[[q_c^{1/N} ]]\bigr)}$. In other words, the parameter $q_c^{1/N}$ at~$c$ on $X(N)(\C )$ 
is actually defined over~$\Z [\zeta_{N}]$, that is~$q_c^{1/N}$ comes from an element of 
the completed local ring~$\hat{\cal O}_{X(N),c}$ of the modular model of~$X(N)$ over 
$\Z [\zeta_{N}]$, along the section~$c$. 
 
Next we describe the local parameters at the cusps on an arbitrary~$X_{G}$. As before, 
let~$\Gamma$ be the pull-back of ${G\cap \SL_2(\Z/N\Z)}$ to ${\Gamma(1)}$. For each 
cusp~$c$ of~$X_G$ we obtain a parameter at~$c$ on~$X_{G}$ by picking a lift 
$\tilde{c}$ of $c$ on $X(N)$ and taking the norm $\prod q_{\tilde{c}}^{1/N}\circ
\gamma$, where~$\gamma$ runs through a set of representatives of $\Gamma /\Gamma (N)$. 
We denote by~$t_{c}$ this parameter in the sequel. Note that it is defined over a (possibly 
strict) subring of $\Z [\zeta_{N}]$.

  It is clear from the definition that the above parameter $t_{c}$ defines a $v$-analytic 
function on a $v$-adic neighborhood of~$c\in X_{G} ({\bar K_{v}})$ which satisfies the 
initial condition ${t_{c} (c)=0}$. Further, if~$e_c$ is the ramification index of the covering 
$X_G \to X(1)$ at~$c$ (note that~$e_c$ divides~$N$) then, setting ${q_{c}:=t_{c}^{e_{c}}
}$, the familiar expansion ${j= {q_{c}}^{-1}+ 744+196884{q_{c}}+\ldots}$ 
holds in a  $v$-adic neighborhood of~$c$, the right-hand side converging $v$-adically.

Let us define explicitly a set ${\Omega_c=\Omega_{c,v}}$ on which~$t_c$ and~$q_c$ are 
defined and analytic. Assume first that~$v$ is archimedean. Let $\tilD$ be the fundamental
domain~$D$ modified as follows: 
\begin{equation}
\label{etild}
\tilD = D\cup\{i\infty\}\smallsetminus (\text{the arc connecting $i$ and $e^{2\pi 
i/3}$}).
\end{equation}
Then the set ${\Delta=\tilD+\Z}$ of translated of~$\tilD$ by the rational integers has the
following properties.

\begin{enumerate}
\item
If for some ${\tau\in\Delta}$ and ${\gamma\in \Gamma(1)}$ we have ${\gamma(\tau)
\in \Delta}$ then ${\tau\equiv\gamma(\tau)\bmod 1}$. 

\item
For any ${\gamma \in \Gamma(1)}$ either ${\gamma(\Delta)=\Delta}$ or 
${\gamma(\Delta)\cap \Delta = \varnothing}$.

\item
The quotient $\Delta/\Gamma$ is an open subset of ${\bar\HH/\Gamma =X_G(\C)}$. (Recall 
that~$\Gamma$ is the pull-back of ${G\cap\SL_2(\Z/N\Z)}$ to $\SL_2(\Z)$.)

\end{enumerate}

The same properties hold for the set $\gamma(\Delta)$ for any ${\gamma\in \Gamma(1)}$. 
Now we define ${\Omega_c = \gamma(\Delta)/\Gamma}$, where ${\gamma\in \Gamma(1)}$ 
is chosen so that $\gamma(i\infty)$ represents the cusp~$c$. It follows from the properties 
above that the sets $\Omega_c$ are pairwise disjoint, and that~$q_c$ and~$t_c$ are defined
and analytic on~$\Omega_c$.

If~$v$ is non-archimedean, then we define ${\Omega_c=\Omega_{c,v}}$ as the set
of the points from $X_G(\bar{K}_{v})$ having reduction~$c$ at~$v$.

\begin{proposition}\label{p31}
Put 
$$
X_G(\bar K_v)^+=
\begin{cases}
\{P\in X_G(\bar K_v) : |j(P)|>3500\} & \text{if $v\in M_K^\infty$},\\
\{P\in X_G(\bar K_v) : |j(P)|>1\} & \text{if $v\in M_K^0$}.
\end{cases}
$$
Then
\begin{equation}
\label{egamom}
X_G (\bar{K}_{v})^+\subseteq\bigcup_{c\in \CC}\Omega_{c,v}
\end{equation}
with equality for the non-archimedean~$v$.
Also, for ${P\in \Omega_{c,v}}$   we have  
\begin{gather}
\label{e1100}
\left|j(P)-q_c(P)^{-1}\right|_v\le 1100,\\
\label{everysimple}
\frac32 |j(P)|_v\ge \left|q_c(P)^{-1}\right|_v\ge \frac12 |j(P)|_v
\end{gather}
if~$v$ is archimedean, and ${|j(P)|_v=|q_c(P)^{-1}|_v}$ if~$v$ is non-archimedean. 
\end{proposition}

\begin{proof}
For the non-archimedean~$v$ both statements are immediate. For archimedean~$v$ they 
follow from Corollary~\ref{cdplus}. Indeed, fix a uniformization  ${X(\bar K_v)=
\bar\HH/\Gamma}$ and let  ${\tau_0\in \bar\HH}$ be a lift of ${P\in X_{G}
(\bar K_v)^+}$.  
Pick ${\gamma\in \Gamma(1)}$ such that ${\tau=\gamma(\tau_0)\in D}$. Then 
${|j(\tau)|=|j(P)|>3500}$, and item~(\ref{i2500}) of Corollary~\ref{cdplus} implies that
${|q_\tau|<0.001<e^{-2\pi}}$. It follows that ${\tau \in \tilD}$, which is equivalent to 
saying that ${P\in \Omega_c}$ with ${c=\gamma^{-1}(c_\infty)}$. This 
proves~(\ref{egamom}), and item~(\ref{iverysimple}) of Corollary~\ref{cdplus} 
implies~(\ref{e1100}) and~(\ref{everysimple}).  \qed
\end{proof}

\bigskip

The proposition implies that for every  ${P\in X_G (\bar{K}_{v})^+}$ there exists a 
cusp~$c$ such that ${P\in \Omega_{c,v}}$. We call it \textsl{a $v$-nearby cusp}, or 
simply \textsl{a nearby cusp} to~$P$. 

\begin{remark}
As we have already seen,  the sets~$\Omega_c$ are pairwise disjoint when~$v$ is 
archimedean. The same is true if~$v$ is non-archimedean and ${v(N)=0}$,  as in this case 
the cusps define a finite \'etale scheme over ${\cal O}_{v}$. Thus, in these case the nearby 
cusp is well-defined.

However, if ${v(N)>0}$ then the sets~$\Omega_c$ are not disjoint, see the addendum below; 
in particular, in this case a point in $X_G(\bar K_v)$ may have several nearby cusps. This is 
absolutely harmless for our arguments, but for further applications it would be nice to refine 
the sets~$\Omega_c$ to make them pairwise disjoint; in particular, this  would  allow us to 
define ``{\it the} $v$-nearest cusp'' rather than just ``{\it a} $v$-nearby cusp'' for any~$v$. In 
the addendum below we examine more carefully geometry of cusps and their neighborhoods 
in the non-archimedean case. In particular, we explicitly exhibit a pairwise disjoint system of 
$v$-adic cusp neighborhoods. 
\end{remark}

\subsection*{Addendum: more on the cusps and their $v$-adic 
neighborhoods\footnote{The material of this addendum will not be used in the present 
article, but we include it for future references.}} 

Recall first of all that the modular interpretation of $X(N)$ associates with 
each cusp a N\'eron polygon~$C$ with~$N$ sides on~$\Z [\zeta_N ]$, endowed with its 
structure of generalized elliptic curve, and enhanced with a basis of~$C[N]\simeq 
\Z /N\Z \times\mu_N =\langle q^{1/N} ,\zeta_N \rangle$ such that the determinant
of this basis is~$1$, and two bases are identified if they are conjugate by the subgroup ${\pm 
U=\pm\left(\topbot10\topbot *1\right)}$ of $\GL_2(\Z/N\Z)$, the
action being ${\left(\topbot\epsilon0\topbot a\epsilon\right):(q^{1/N},\zeta_N )
\mapsto (q^{\epsilon /N}\zeta_N^a ,\zeta_N^{
\epsilon})}$, for ${\epsilon =\pm 1}$ and ${a\in \Z /N\Z}$. We may, for instance, 
interpret $c_{\infty}$ as the orbit ${\left\{ (C,(q^{\epsilon/N}\zeta_N^a ,
\zeta_{N}^{\epsilon}
)),\ \epsilon\in\{\pm 1\} ,\ a\in\Z /N\Z\right\}}$ of enhanced N\'eron polygons 
over $\Z [\zeta_{N}]$.

It follows that the  modular interpretation of~$X_{G}$ associates to each cusp an orbit of 
our enhanced N\'eron polygon ${\left(C,(q^{1/N},\zeta_N )\right)}$ under the action 
of the group generated by~$G$ and~$\pm U$. We see from the above that the cusps 
of~$X_G$ have values in a subring of~$\Z [\zeta_{N}]$. Assume moreover that~$X_{G}$ 
is defined over~$K$, of which~$v$ is a place of characteristic~$p$, with~$N=p^n N'$ 
and~$p\nmid N'$. Extending $v$ to a place of $\OO_v [\zeta_{N'}]$ if necessary, and
setting ${\OO'_v :=(\OO [\zeta_{N'}])_v}$, one sees that the closed subscheme of cusps 
over~$\OO'_v$ may be written as a sum of connected components of shape $\mathrm{Spec}
(R)$ where~$R$ is a subring of ${\OO'_v [\zeta_{p^n}]}$. Therefore if ${v(N)=0}$, the 
subscheme of cusps is \'etale over~$\OO_v$, but this may not be the case if $v(N)>0$. In the
latter case, however, the ramification is well controlled. Indeed, with the preceding notations, 
set ${\pi :=(1-\zeta_{p^n})}$. Any two different $p^n$-th roots of unity~$\zeta_{p^n}^a$ 
and~$\zeta_{p^n}^b$ satisfy ${(\zeta_{p^n}^a -\zeta_{p^n}^b )=\pi^{p^k} \alpha}$ 
with~$\alpha$ a~$v$-invertible element and ${0\le k\le n-1}$. As $v(\pi )=1/p^{n-1} 
(p-1)$ (normalizing $v$ so that ${v(p)=1}$), it follows that the N\'eron 
polygons enhanced with a level-$N$ structure are distinct over ${\Z [\zeta_N ]/(\pi^{p^{n-1}
+1})}$. The modular 
interpretation shows more precisely that if two different cusps~$c_1$ and~$c_2$ have same 
reduction at~$v$, then~$t_{c_1} (c_2 )$ has $v$-adic valuation less or equal to~$1/(p-1)$ 
(where $t_c$ is the parameter defined at the beginning of this Section). This remark will be 
used later on.  

   To illustrate all this with a familiar example, letting~$G:=\left( \topbot1*\topbot 0*
\right) \subset \GL_2 (\Z /N\Z )$, which gives rise to the modular curve $X_1 (N)$, one
finds that there are $\left|(\Z /N\Z)^\times\right|$ cusps, with modular interpretation 
corresponding to ${\left\{ (C,\zeta_N^{\epsilon a} ): \epsilon\in \{\pm1\} 
\right\}}$ where~$a$ runs through ${(\Z  /N\Z )^\times /\pm 1}$, and 
${\{(C,q^{\epsilon a /N}\zeta_N^{\alpha}) : \epsilon\in \{\pm1\},\ \alpha\in (\Z
/N\Z )\}}$, where~$a$ runs through the same set. The curve $X_1 (N)$ is defined over~$\Q$
and has a modular model over~$\Z$. The cusps in the former subset above have values in 
${\Z \left[\zeta_{N} +\zeta_{N}^{-1}\right]}$, and the cusps in the latter subset have 
values in~$\Z$. In other words, the closed subscheme of cusps over~$\Z$ is isomorphic to 
the disjoint union of ${\mathrm{Spec} \left(\Z \left[\zeta_{N} +\zeta_{N}^{-1}
\right]\right)}$ and $\left|(\Z /N\Z)^\times\right| /2$ copies of $\mathrm{Spec} 
(\Z )$.

Let us examine the $v$-adic neighborhoods of the cusps. If $v(N)=0$ then the cusps of 
$X_{G}$ define a finite \'etale closed 
subscheme of $X_{G}$ over ${\cal O}_{v}$, so the sets~$\Omega_c$ are 
obviously pairwise disjoint.

Now assume that  ${v(N)>0}$. Let~$p$ be the residue characteristic of~$v$ and~$p^n \| N$ 
be the largest power of~$p$ dividing~$N=p^n N'$. The scheme of cusps 
on~$X_{G}$ may be no longer \'etale over~$\OO_v$. We can however still partition it into 
connected components, which totally ramify in the fiber at~$v$. More precisely, setting as 
above~${\OO'}_{v} :=(\OO [\zeta_{N'}])_v$, each connected component over $\OO'_{v}$
is schematically a~$\mathrm{Spec} (R)$ for~$R$ a subring of~${\OO'}_v [\zeta_{p^n}]$. 
Each set~$\Omega_{c}$ as in Proposition~\ref{p31} contains exactly one such connected 
component of cusps, so 
when~$R$ does ramify nontrivially at~$v$, then~$\Omega_{c}$ is clearly ``too large'' (one 
has~$\Omega_{c_{1}} =\Omega_{c_{2}}$ exactly when~$c_{1}$ and~$c_{2}$ have same 
reduction at~$v$). We want to show that, nevertheless, the refined sets $\Omega_c^+$, 
defined by 
$$
\Omega_c^+=\Omega_{c,v}^+ =\bigl\{ P\in \Omega_c:|q_c(P)|_v <p^{-N/(p-1)}\bigr\}
$$
are pairwise disjoint.

If the cusps~$c_1$ and~$c_2$ belong to distinct connected components, then 
already~$\Omega_{c_{1}}$ and~$\Omega_{c_{2}}$ are disjoint, so~$\Omega_{c_{1}}^+$ 
and~$\Omega_{c_{2}}^+$ are disjoint \emph{a fortiori}. Now assume that ~$c_1$ 
and~$c_2$ belong to the same component, {\it{i.e.}}, have same reduction at~$v$. In this 
case, as explained above, one may write ${t_{c_{1}}(c_{2} )=\pi^{p^k} a\in 
\OO'_v [\zeta_{p^n} ]}$, for $\pi$ a certain uniformizer (e.g. $\pi :=(\zeta_{p^n}-1)$), 
where the element~$a$ is~$v$-invertible and ${0\le k\le n-1}$. As   
$\Omega_{c_1}^+$ is contained in $\left\{ P\in X_{G} (\bar{K}_v ):|t_{c_1}(P)|_v <
p^{-1/(p-1)}\right\}$ (recall $q_c =t_c^{e_c}$, with $e_c |N$), we see that~$c_2$ does not
belong to~$\Omega_{c_1}^+$, which implies that the sets~$\Omega_{c_{1}}^+$ 
and~$\Omega_{c_{2}}^+$ are disjoint.

\section{Modular Units}
\addtocontents{toc}{\vspace{-0.7\baselineskip}}
\label{smuni}

In this section we  recall the construction of modular units on the modular curve $X_G$. By 
a \emph{modular unit} we mean a rational function on~$X_G$ having poles and zeros only 
at the cusps. 

\subsection{Integrality of Siegel's Function}

For ${\bfa\in \Q^2\smallsetminus\Z^2}$ Siegel's function~$g_\bfa$ (see 
Subsection~\ref{sssieg}) is algebraic over the field $\C(j)$: this follows from the fact that 
$g_\bfa^{12}$ is automorphic of certain  level \cite[page~29]{KL81}. Since~$g_\bfa$ 
is holomorphic and does not vanish on the upper half-plane~$\HH$, both~$g_\bfa$ and 
$g_\bfa^{-1}$  must be integral over the ring $\C[j]$. Actually, a stronger assertion holds.

\begin{proposition}
\label{psiu}
Both~$g_\bfa$ and   ${\left(1-\zeta_N\right)g_\bfa^{-1}}$ are integral over  $\Z[j]$. 
Here~$N$ is the exact order of~$\bfa$ in $(\Q/\Z)^2$ and~$\zeta_N$ is a primitive $N$-th 
root of unity. 
\end{proposition}

This  is, essentially, established in~\cite{KL81}, but is not stated explicitly therein. For a 
complete proof, see~\cite[Proposition~2.2]{BP09}.

\subsection{Modular Units on $X(N)$}
\label{ssxn}

From now on, we fix an integer ${N>1}$. Recall that the curve $X(N)$ is defined over the 
field $\Q(\zeta_N)$. Moreover, the field ${\Q\bigl(X(N)\bigr)=\Q(\zeta_N)\bigl(X(N)
\bigr)}$  is a Galois extension of $\Q(j)$, the Galois group being isomorphic to $\GL_2
(\Z/N\Z)/\pm 1$. The isomorphism 
\begin{equation}
\label{eisog}
\gal\left(\left.\Q\bigl(X(N)\bigr)\right/\Q(j)\right)\cong \GL_2(\Z/N\Z)/\pm 1
\end{equation}
is defined up to an inner automorphism; once it is fixed, we have the well-defined 
isomorphisms 
\begin{equation}
\gal\left(\left.\Q\bigl(X(N)\bigr)\right/\Q(\zeta_N,j)\right)\cong
\SL_2(\Z/N\Z)/\pm 1,\qquad
\label{eisoc}
\gal\bigl(\left.\Q(\zeta_N)\right/\Q\bigr)\cong (\Z/N\Z)^\times
\end{equation}
(which give the geometric and arithmetic parts of the Galois group respectively). We may
identify the groups on the left and on the right in~(\ref{eisog} and~\ref{eisoc}). Our 
choice of the isomorphism~(\ref{eisog}) will be specified in Proposition~\ref{pua}.

According to Theorem~1.2 from \cite[Section~2.1]{KL81}, given ${\bfa=(a_1,a_2) \in
(N^{-1}\Z)^2\smallsetminus \Z^2}$, the function $g_\bfa^{12N}$ is $\Gamma
(N)$-automorphic of weight~$0$. Hence $g_\bfa^{12N}$ defines a rational function on the 
modular curve $X(N)$, to be denoted  by~$u_\bfa$ (one checks that~$u_\bfa =u_{-\bfa}$).
Since the root of 
unity in~(\ref{eaa'}) is of order dividing $12N$, we have ${u_\bfa=u_{\bfa'}}$ when 
${\bfa\equiv \bfa'\mod \Z^2}$. Hence~$u_\bfa$ is well-defined when~$\bfa$ is a 
non-zero element of the abelian group ${(N^{-1}\Z/\Z)^2}$, which will be assumed in the 
sequel. We put
${\bfA=(N^{-1}\Z/\Z)^2\smallsetminus\{0\}}$.

The functions~$u_\bfa$ have the following properties. 

\begin{proposition}
\label{pua}
\begin{enumerate}
\item
\label{iintmun}
The functions~$u_\bfa$ and ${(1-\zeta_{N_\bfa})^{12N}u_\bfa^{-1}}$ are integral 
over $\Z[j]$, where~$N_\bfa$ is the exact order of~$\bfa$ in ${(N^{-1}\Z/\Z)^2}$. 
In particular,~$u_\bfa$ has zeros and poles only at the cusps of $X(N)$.

\item
\label{ifrick}
The functions~$u_\bfa$ belong to the field ${\Q\bigl(X(N)\bigr)}$,
and the Galois action on the set $\{u_\bfa\}$ is compatible with the (right) 
linear action of $\GL_2(\Z/N\Z)$ on~$\bfA$  in the following sense: the 
isomorphism~(\ref{eisog}) can be chosen so that for any ${\sigma\in \gal\Bigl(\left.
\Q\bigl(X(N)\bigr)\right/\Q(j)\Bigr)=\GL_2 (\Z/N\Z)/\pm 1}$ and any ${\bfa\in 
\bfA}$  we have ${u_\bfa^\sigma=u_{\bfa\sigma}}$.

\item
\label{iord}
For the cusp~$c_\infty$ at infinity we have ${\ord_{c_\infty}u_\bfa=12N^2
\ell_\bfa}$, where~$\ell_\bfa$ is defined in Subsection~\ref{ssnota}. For an arbitrary 
cusp~$c$ we have ${\left|\ord_cu_\bfa\right|\le N^2}$.

\end{enumerate}

\end{proposition}

\begin{proof}
Item~(\ref{iintmun}) follows from Proposition~\ref{psiu}. Item~(\ref{ifrick}) is 
Proposition~1.3 from \cite[Chapter~2]{KL81}. We are left with item~(\ref{iord}). The 
$q$-order of vanishing of~$u_\bfa$ at $i\infty$  is $12N\ell_\bfa$. Since the 
ramification index of the morphism ${X(N)\to X(1)}$ at every cusp is~$N$, we obtain 
${\ord_{c_\infty}u_\bfa=12N^2\ell_\bfa}$. Since ${|\ell_\bfa|\le 1/12}$, we have 
${\left|\ord_{c_\infty}u_\bfa
\right|\le N^2}$. The case of arbitrary~$c$ reduces to the case ${c=c_\infty}$ upon 
replacing~$\bfa$ by $\bfa\sigma$ where  ${\sigma\in \GL_2(\Z/N\Z)}$ is such that 
${\sigma(c)=c_\infty}$.    \qed
\end{proof}

\bigskip

The group generated by the principal divisors ${(u_\bfa)}$, where ${\bfa\in \bfA}$, is 
contained in  the group of cuspidal divisors on~$X(N)$ (that is, the divisors  supported at the
set ${\CC(N)=\CC(\Gamma(N))}$ of cusps). Since principal divisors are of degree~$0$, the
rank of the former group is at most ${|\CC(N)|-1}$. It is fundamental for us that this rank is 
indeed maximal.  The following proposition is Theorem~3.1 in 
\cite[Chapter~2]{KL81}.

\begin{proposition}
\label{pmdrxn}
The  group generated by the set
${\left\{(u_\bfa):\bfa\in \bfA\right\}}$ is of rank
${|\CC(N)|-1}$. \qed
\end{proposition}

We also need to know the  behavior of the functions~$u_\bfa$ near the cusps, and estimate 
them in terms of the modular invariant~$j$. In the following proposition~$K$ is a number 
field containing~$\zeta_N$ and~$v$ is a valuation of~$K$, extended somehow to~$\bar 
K$. For ${v\in M_K}$ we define
\begin{equation}
\label{erhov}
\rho_v=
\begin{cases}
12N\log N & \text{if ${v\in M_K^\infty}$},\\
0& \text{if ${v\in M_K^0}$ and ${v(N)=0}$,}\\
\displaystyle 12N \frac{\log p}{p-1} & \text{if ${v\in M_K^0}$ and ${v\mid p
\mid N}$}.
\end{cases}
\end{equation}
We use the notation of Section~\ref{snecu}. 

\begin{proposition}
\label{ploc}
Let~$c$ be a cusp of $X(N)$ and~$v$ a place of~$K$. 
For ${P\in \Omega_{c,v}}$  we have
$$
\bigl|\log |u_\bfa(P)|_v- \ord_cu_\bfa\log|t_c(P)|_v\bigr| \le \rho_v.
$$
For ${v\in M_K^\infty}$ and ${P\in X(N)(K_v)}$  we have 
$$
\bigl|\log |u_\bfa(P)|_v\bigr|\le 
N\log\bigl(|j(P)|_v+2400\bigr)+\rho_v. 
$$ 
\end{proposition}

\begin{proof}
The first statement for ${c=c_\infty}$  is an immediate consequence of 
Propositions~\ref{psiar} and~\ref{psinar} (notice that ${\log| q_c|_v=N
\log |t_c|_v}$ for every cusp~$c$). The general case reduces 
to the case ${c=c_\infty}$ by applying a suitable Galois automorphism. The second 
statement follows from Corollary~\ref{cgalar}.  \qed
\end{proof}

\subsection{$K$-Rational Modular Units on $X_G$}
\label{ssmunk}

Now let~$K$ be a number field, and let~$G$ be a subgroup of $\GL_2(\Z/N\Z)$ containing 
$-1$. Let $\det G$ be the image of~$G$ under the determinant map ${\det \colon \GL_2(
\Z/N\Z)\to (\Z/N\Z)^\times=\gal(\Q(\zeta_N)/\Q)}$ (recall that we have a well-defined
isomorphism~(\ref{eisoc})). In the sequel we shall assume that ${K\supseteq\Q
(\zeta_N)^{\det G}}$, where $\Q(\zeta_N)^{\det G}$ is the subfield of $\Q(\zeta_N)$ 
stable under $\det G$. This assumption implies that the curve~$X_G$ is defined over~$K$. 
Then ${G':=\gal\left(\left.K\bigl(X(N)\bigr)\right/K\left(X_G \right)\right)}$
is a subgroup of~$G/\pm 1$, which contains the geometric part~$(G\cap \SL_{2} (\Z /N\Z 
))/\pm 1$. For every ${\bfa\in \bfA}$ we put ${w_\bfa= \prod_{\sigma\in G'}
u_{\bfa \sigma}}$. Since ${u_{\bfa \sigma}=u_\bfa^\sigma}$, the functions~$w_\bfa$
are contained in $K(X_G)$.  They have the following properties. 

\begin{proposition}
\label{pwa}
\begin{enumerate}
\item
\label{imunk}
The functions~$w_\bfa$ have zeros and poles only at the cusps of $X_G$. If~$c$ is such a 
cusp, then ${\left|\ord_cw_\bfa\right|\le |G'|N^2}$. 
\item
\label{iintk}
For every ${\bfa \in \bfA}$ there exists an algebraic integer ${\lambda_\bfa \in \Z
[\zeta_N]}$, which is a product of $|G'|$ factors of the form ${\left(1-\zeta_{N'}
\right)^{12N}}$, where ${N'\mid N}$, such that the functions~$w_\bfa$ and 
${\lambda_\bfa w_\bfa^{-1}}$ are integral over~$\Z[j]$. 

\item
\label{iadd}
Let~$c$ be a cusp of~$X_G$. Then for ${v\in M_K}$ and ${P\in \Omega_{c,v}}$ we have 
$$
\bigl|\log\left|w_\bfa(P)\right|_v- \ord_cw_\bfa\log|t_c(P)|_v
\bigr| \le |G'|\rho_v. 
$$

\item
\label{icrk}
For ${v\in M_K^\infty}$   and ${P\in X_G(K_v)}$ we have 
$$
\bigl|\log |w_\bfa(P)|_v\bigr|\le 
|G'|N \log\bigl(|j(P)|_v+2400\bigr)+|G'|\rho_v.
$$ 

\item
\label{imdk}
The group generated by the principal divisors $(w_\bfa)$ is of rank ${|\CC(G,K)|-1}$. 

\end{enumerate}
\end{proposition}

\begin{proof}
Items~(\ref{imunk}) and~(\ref{iintk})  follow from  Proposition~\ref{pua}, 
items~(\ref{icrk}) and~(\ref{iadd})  follow from Proposition~\ref{ploc}. Finally, 
item~(\ref{imdk}) follows from Proposition~\ref{pmdrxn}. \qed
\end{proof}

\subsection{A Unit Vanishing at the Given Cusps}
\label{ss43}
Item~(\ref{imdk}) of Proposition~\ref{pwa} implies that for any proper subset~$\Sigma$ of 
$\CC(G,K)$ there is a $K$-rational unit on~$X_G$  vanishing at this subset; moreover, such a unit can be expresses as a multiplicative combination of the units~$w_\bfa$. We call it a \textsl{Runge unit} for~$\Sigma$. In this 
subsection we give a quantitative version of this fact. 
We shall use the following simple lemma, where we denote by $\|\cdot\|_1$ the 
$\ell_1$-norm.

\begin{lemma}
\label{llia}
Let~$M$ be an  ${s\times t}$ matrix of rank~$s$ with entries in~$\Z$. Assume that the 
entries of~$M$ do not exceed~$A$ in absolute value. Then there exists a vector  ${\bfb\in 
\Z^t}$ such that ${\|\bfb \|_1\le s^{s/2+1}A^{s-1}} $,
and such that all the~$s$ coordinates of the vector $M\bfb$ (in the standard basis) are strictly 
positive. 
\end{lemma}

\begin{proof}
Assume first that ${s=t}$. Let~$d$ be the determinant of~$M$. Then the column vector 
${(|d|, \ldots,|d|)}$ can be written as $M\bfb$, where ${\bfb=(b_1, \ldots, b_s)}$ 
with~$b_k$ being (up to the sign) the determinant of the matrix obtained from~$M$ upon 
replacing the $k$-th column by ${(1, \dots,1)}$. Using Hadamard's inequality, we
bound $|b_k|$ by $s^{s/2}A^{s-1}$. This proves the lemma in the case ${s=t}$.
The general case reduces to the case ${s=t}$ by selecting a non-singular ${s\times s}$ 
sub-matrix, which gives~$s$ entries of the vector~$\bfb$; the remaining ${t-s}$ entries are
set to be~$0$.  \qed
\end{proof}

\bigskip

Now let~$G$,~$K$ and~$G'$ be as in Subsection~\ref{ssmunk}.

\begin{proposition}
\label{pw}
Let~$\Sigma$ be a proper subset of $\CC(G,K)$ and~$s$ a positive integer satisfying ${s\ge 
|\Sigma|}$. Then one can associate to every ${\bfa \in \bfA}$ an integer~$b_\bfa$ such 
that 
\begin{equation}
\label{ebbb}
B:=\sum_{\bfa\in \bfA}|b_\bfa| \le  s^{s/2+1}\left(|G'|N^2\right)^{s-1}
\end{equation} 
and the unit ${w:=\prod_{\bfa \in \bfA}w_\bfa^{b_\bfa}}$ has the following properties.

\begin{enumerate}
\item If~$c$ is a cusp such that the orbit of~$c$ is in~$\Sigma$ then ${\ord_cw>0}$ (that 
is,~$w$ is a Runge unit for~$\Sigma$).

\item
\label{ilamb} There exists an algebraic integer~$\lambda$, which is a product of at most 
$|G'|B$ factors of the form ${\left(1-\zeta_{N'}\right)^{12N}}$, where ${N'\mid N}$, 
such that $\lambda w$ is integral over $\Z[j]$.  

\item
\label{iaddb}
Let~$c$ be a cusp of~$X_G$. Then for ${v\in M_K}$ and ${P\in \Omega_{c,v}}$ we have 
$$
\bigl|\log\left|w(P)\right|_v- \ord_cw\log|t_c(P)|_v
\bigr| \le 
B |G'|\rho_v .
$$

\item
\label{icrkb}
For ${v\in M_K^\infty}$   and ${P\in X_G(K_v)}$ we have 
$$
\bigl|\log |w(P)|_v\bigr|\le 
B |G'|N \log\bigl(|j(P)|_v+2400\bigr)+B|G'|\rho_v .
$$

\end{enumerate}

\end{proposition}

\begin{proof}
The $K$-rational Galois orbit of a cusp~$c$ has ${[K(c):K]}$ elements. Fix a representative 
in every such orbit  and consider the ${|\CC(G,K)|\times|\bfA| }$ matrix $\left(\ord_c 
w_\bfa\right)$, where~$c$ runs over the set of selected representatives. According to 
item~(\ref{imdk}) of Proposition~\ref{pwa}, this matrix is of rank ${|\CC(G,K)|-1}$, and 
the only (up to proportionality) linear relation between the rows  is ${\sum_c[K(c):K]
\ord_cw_\bfa=0}$ for  every ${\bfa\in \bfA}$. It follows that any proper subset of the 
rows of our matrix is linearly independent. In particular, if we select  the rows corresponding
to the set~$\Sigma$, we get a sub-matrix  of rank~$|\Sigma|$. Applying to it 
Lemma~\ref{llia}, where we may take ${A=|G'|N^2}$ due to item~(\ref{imunk}) of 
Proposition~\ref{pwa}, we find integers~$b_\bfa$  such that~(\ref{ebbb}) holds and the  
function ${w=\prod_{\bfa\in \bfA} w_\bfa^{b_\bfa}}$ is as desired, by 
Proposition~\ref{pwa}, (\ref{iintk})--(\ref{icrk}). 
\qed \end{proof}

\section{Proof of Theorem~\ref{tbo}}
\addtocontents{toc}{\vspace{-0.7\baselineskip}}
\label{sproof}

We use the notations of Section~\ref{snecu}. We also write ${d_v=[K_v:\Q_v]}$ and 
${d=[K:\Q]}$. We fix an extension of every ${v\in M_K}$ to~$\bar K$ and denote this 
extension by~$v$ as well. We shall use the following obvious estimates for the 
quantities~$\rho_v$ defined in~(\ref{erhov}): 
\begin{equation}
\label{ecalrest}
\sum_{v\in M_K^\infty}d_v\rho_v= 12dN\log N, \qquad \sum_{v\in M_K^0}d_v
\rho_v = 12dN\sum_{p\mid N}\frac{\log p}{p-1}\le   12dN\log N. 
\end{equation}

\subsection{A Runge Unit}
\label{ssRungeUnit}
Fix ${P\in Y_G(\OO_S)}$. Let~$S_1$ consist of the places ${v\in M_K}$ such that ${P\in
X_G(K_v)^+}$. Plainly, ${S_1\subset S}$. For ${v\in S_1}$ let~$c_v$ be a $v$-nearby 
cusp to~$P$ (if there are several, choose any of them) and let~$\Sigma$ be the set of all 
$\gal (\bar K/K)$-orbits of cusps containing some of the~$c_v$. Then ${|\Sigma|\le |S_1|
\le |S|}$, and since ${|S|< |\CC(G,K)|}$ by assumption,~$\Sigma$ is a proper subset of
$\CC(G,K)$. It follows from Proposition~\ref{pw} that there exists a $K$-rational modular 
unit ${w=\prod_{\bfa\in 
\bfA} w_\bfa^{b_\bfa}}$ such that ${\ord_{c_v}w>0}$ for every ${v\in S_1}$ (a 
\textsl{Runge unit}) for which ${B:=\sum_{\bfa\in \bfA}|b_\bfa|}$ 
satisfies~(\ref{ebbb}),  where we may put ${s=|S|}$.

Since~$w$ is a modular unit and~$P$ is not a cusp, we have ${w(P)\ne 0, \infty}$, and the
product formula gives ${\sum_{v\in M_K}d_v\log |w(P)|_v=0}$. We want to show that 
this is impossible when $\height(P)$ is too large.

\subsection{Partitioning the Places of~$K$}
\label{sspart}
We partition the set of places~$M_K$ into three pairwise disjoint subsets: ${M_K=S_1\cup 
S_2\cup S_3}$, where ${S_i\cap S_j =\varnothing}$ for ${i\ne j}$. The set~$S_1$ is 
already defined. The set~$S_2$ consists of the \emph{archimedean} places not belonging 
to~$S_1$. Obviously, ${S_1\cup S_2\subset S}$. Finally, the set~$S_3$ consists of the 
places not belonging to ${S_1\cup S_2}$; in other words, ${v\in S_3}$ if and only if~$v$ 
is non-archimedean and ${|j(P)|_v\le 1}$. 

We will estimate from above the quantities 
$$
\Xi_i=d^{-1}\sum_{v\in S_i}d_v\log |w(P)|_v \qquad (i=1,2,3).
$$
We will show that ${\Xi_1 \le -N^{-1}d\height(P)+O(1)}$, where the $O(1)$-term is 
independent of~$P$ (it will be made explicit). Further, we will bound $\Xi_2$ and $\Xi_3$
independently of~$P$. Since 
\begin{equation}
\label{esss}
\Xi_1+\Xi_2+\Xi_3=0,
\end{equation} 
this would bound $\height(P)$.

\subsection{Estimating $\Xi_1$}
\label{ssxi1}
For ${v\in S_1}$ we have ${P\in \Omega_{c_v,v}}$, so we may apply item~(\ref{iaddb}) 
of Proposition~\ref{pw}. Since ${\ord_{c_v}w\ge 1}$ and ${\log |q_{c_v}(P)|_{v}= 
e_{c_v}\log |t_{c_v} (P)|_{v}}$ with ${e_{c_v}\mid N}$, we have
\begin{align}
\label{elogmin0}
\Xi_1 & \le d^{-1}\sum_{v\in S_1} d_v \frac{\ord_{c_v}w}{e_{c_v}}\log|q_{c_v}(P)|_v + B|G'|d^{-1}\sum_{v\in M_K}d_v  \rho_v \\
\label{elogmin1}
&\le N^{-1}d^{-1}\sum_{v\in S_1} d_v \log|q_{c_v}(P)|_v +24 B|G'|N \log N \\
\label{elogmin}
&\le -N^{-1}d^{-1}\sum_{v\in S_1}d_v \log|j(P)|_v +N^{-1}\log 2+24 B|G'|N\log N,
\end{align}
where we use~(\ref{ecalrest}) and~(\ref{everysimple}).
Further,  for ${v\in M_K\smallsetminus S_1}$ we have ${|j(P)|_v\le 3500}$ if~$v$ is 
archimedean, and  ${|j(P)|_v\le 1}$ if~$v$ is non-archimedean. It follows that
$$
d^{-1}\sum_{v\in M_K\smallsetminus S_1}d_v\log^+|j(P)|_v \le 
\log 3500.
$$
Since ${|j(P)|_v>1}$ for ${v\in S_1}$, one may also replace $\log|j(P)|_v$ by 
$\log^+|j(P)|_v$ in~(\ref{elogmin}). Hence
\begin{align}
\Xi_1&\le -N^{-1}d^{-1}\sum_{v\in M_K}d_v\log^+|j(P)|_v+N^{-1}\log 7000+
24 B|G'|N\log N \nonumber\\
\label{esone}
&=-N^{-1}\height(P)+N^{-1}\log 7000+  24 B|G'|N\log N.
\end{align}

\subsection{Estimating $\Xi_2$, $\Xi_3$ and Completing the Proof}
\label{ssxi2}
Item~(\ref{icrkb}) of Proposition~\ref{pw} together with~(\ref{ecalrest}) implies that 
\begin{equation}
\label{estwo}
\Xi_2\le B|G'|N\log 5900+ 12B|G'|N\log N 
\end{equation}

Further, let~$\lambda$ be from item~(\ref{ilamb}) of Proposition~\ref{pw}. Then
${\height(\lambda)\le 12B|G'|N\log2}$,
because ${\height(1-\zeta)\le \log2}$ for a root of unity~$\zeta$. 
For ${v\in S_3}$ the number $j(P)$ is a $v$-adic integer. Hence so is the number $\lambda 
w(P)$. It follows that ${|w(P)|_v\le |\lambda^{-1}|_v}$ for ${v\in S_3}$, and 
$$
\Xi_3\le d^{-1}\sum_{v\in S_3}d_v\log \left|\lambda^{-1}\right|_v\le\height
(\lambda^{-1})=\height(\lambda) \le 12B|G'|N\log 2. 
$$
Combining this with~(\ref{esss}),~(\ref{esone}) and~(\ref{estwo}), we obtain
${\height(P)\le 36 B|G'|N^2\log 2N}$,
which, together with~(\ref{ebbb})  implies~(\ref{etbo}) with $|G|/2$ replaced by $|G'|$. 

If ${S=M_K^\infty}$ then in the second sum in~(\ref{elogmin0}) one can replace 
${v\in M_K}$ by ${v\in M_K^\infty}$. Hence in~(\ref{elogmin1}),~(\ref{elogmin}) 
and~(\ref{esone}) one may replace~$24$ by~$12$, which allows ${\height(P)\le 
24B|G'|N^2\log 3N}$. This completes the proof of the theorem. \qed

\section{Special Case: the Split Cartan Group}
\addtocontents{toc}{\vspace{-0.7\baselineskip}}
\label{sspc}

When~$G$ is a particular group, one can usually obtain a stronger result than in general. For
instance in~\cite{BP09} we examined the case when ${N=p}$ is a prime number,~$G$ is 
the normalizer of a split Cartan subgroup of $\GL_2(\Z/p\Z)$ and ${K=\Q}$. In this case we 
obtained the estimate ${\height(P)\le C\sqrt p}$ for ${P\in Y_G(\Z)}$, with an absolute 
constant~$C$.  In this section we consider the case when ${N=p}$ is a prime number,~$G$ is 
the Cartan subgroup itself, and~$K$ is a quadratic field. (Note that the fact we are not 
working over $\Q$ any longer makes a significant difference with the normalizer-of-Cartan 
case studied in~\cite{BP09} (appearance of the $\Xi_{2}$-term), and requires much of the 
generality of the first part of the present work.) Without loss of generality we may 
assume that~$G$ is the diagonal subgroup of  $\GL_2(\F_p)$. The modular curve~$X_G$, 
corresponding to this subgroup, is denoted by $X_\splic(p)$. It parametrizes geometric 
isomorphism classes $\bigl(E,(A,B)\bigr)$ of elliptic curves endowed with an ordered pair 
of $p$-isogenies. There is an isomorphism $\phi\colon X_{0} (p^2 )\to X_\splic(p)$ 
over~$\Q$ defined functorially as
$$
\bigl(E,A_{p^2}\bigr)\mapsto \bigl(E/pA_{p^2} , (A_{p^2}/pA_{p^2},E[p]/p
A_{p^2})\bigr).
$$
On the Poincar\'e upper half-plane~$\HH$, the map~$\phi$ is induced by the map ${\tau
\mapsto p\tau}$.

The homographic action of the matrix $\left(\begin{smallmatrix}0&1\\1&0
\end{smallmatrix}\right)$ on~$\HH$ defines an involution of~$X_\splic$, 
which modularly is $(E,(A,B))\mapsto (E,(B,A))$. 

\begin{theorem}
\label{tspc}
Let ${p\ge 3}$ be a prime number and~$K$ a number field of degree at most~$2$. Then for 
any ${P\in Y_\splic(p)(\OO_K)}$ we have ${\height(P)\le 24p\log 3p}$. 
\end{theorem}

\begin{remark}
With some little additional effort, one can obtain a stronger estimate ${\height(P)\le Cp}$ 
(and probably even ${\height(P)\le C\sqrt p}$) with an absolute constant~$C$. However, 
the bound of Theorem~\ref{tspc} is easier to obtain and sufficient for our purposes. 
\end{remark}

The curve $X_\splic(p)$ has ${p+1}$ cusps ${c_\infty, c_0, c_1, \ldots, c_{p-1}}$ 
corresponding, respectively, to the points ${i\infty, 0, 1/p, \ldots, (p-1)/p}$ of~$\bar\HH$.
The morphism ${X(p)\to X_\splic(p)}$ is unramified at the cusps, and the morphism 
${X_\splic\to X(1)}$ is ramified with index~$p$ at all the cusps (as can be immediately
seen by computing ramification indices), so the local parameter at every cusp~$c$ of 
$X_\splic$ is ${t_c=q_c^{1/p}}$.

The cusps  $c_\infty$ and~$c_0$ are defined over~$\Q$. The cusps ${c_1, \ldots, c_{p-1}}$ 
are defined over the cyclotomic field $\Q(\zeta_p)$ and are conjugate over~$\Q$. Thus the 
set $\CC(G,\Q)$ consists of $3$~elements, and the group  generated by the principal divisors 
$(w_\bfa)$ is of rank~$2$ by Proposition~\ref{pwa}, item~(\ref{imdk}). Moreover, it is 
clearly generated by the divisors~$(w_\bfa)$ with ${\bfa \in  \bigl\{(1/p,0), (0,1/p)
\bigr\}}$. More precisely, we have the following.

\begin{proposition}
\label{pprindiv}
The principal divisors~$(w_\bfa)$ with ${\bfa \in  \bigl\{(1/p,0), (0,1/p) \bigr\}}$ are
\begin{align}
(w_{(1/p,0)})&= -\frac12p(p-1)^2 (c_{\infty} -pc_{0} +c_1+\cdots+c_{p-1}), \\
(w_{(0,1/p)})&= -\frac12p(p-1)^2(-pc_{\infty} +c_{0} +c_1+\cdots+c_{p-1}).
\end{align}
\end{proposition}

\begin{proof}
Denote by ${p^{-1}\F_p^\times}$ the set of non-zero elements of ${p^{-1}\Z/\Z}$. 
The $G'$-orbits of $(1/p,0)$ and $(0,1/p)$ are ${\bigl\{(a,0): a\in p^{-1}\F_p^\times
/\pm 1\bigr\}}$ and ${\bigl\{(0,a): a\in p^{-1}\F_p^\times /\pm 1\bigr\}}$ 
respectively, each element of each orbit occurring exactly ${p-1}$ times. Since the 
morphism ${X(p)\to X_\splic}$ is unramified at the cusps, we have
$$
\ord_{c_\infty}w_{(1/p,0)}= (p-1) \sum_{a\in p^{-1}\F_p^\times /\pm 1}
\ord_{c_\infty}u_{(a,0)} = 12p^2(p-1) \sum_{a\in p^{-1}\F_p^\times  /\pm 1} 
\ell_{(a,0)}
$$
by item~(\ref{iord}) of  Proposition~\ref{pua}. It follows that
$$
\ord_{c_\infty}w_{(1/p,0)}=3p^2(p-1)\sum_{k=1}^{p-1}B_2\left(\frac kp\right)
= -\frac12p(p-1)^2,
$$
where we use the identity
$$
\sum_{k=1}^N B_2\left(\frac kN\right)= -\frac{N-1}{6N}
$$
for any positive integer~$N$. 
In a similar fashion,
$$
\ord_{c_\infty}w_{(0,1/p)} =12p^2(p-1) \sum_{a\in p^{-1}\F_p^\times  /\pm 1} 
\ell_{(0,a)} =3p^2(p-1)^2B_2(0)= \frac12p^2(p-1)^2.
$$
Further, the involution induced by the matrix $\left(\begin{smallmatrix}0&1\\1&0
\end{smallmatrix}\right)$ defined before Theorem~\ref{tspc} exchanges the cusps~$c_0$ 
and~$c_\infty$ and the units $w_{(1/p,0)}$ and $w_{(0,1/p)}$. It follows that 
$$
\ord_{c_0}w_{(1/p,0)}=\ord_{c_\infty}w_{(0,1/p)}= p^2(p-1)^2 /2, \qquad 
\ord_{c_0}w_{(0,1/p)}=\ord_{c_\infty}w_{(1/p,0)}= -p(p-1)^2 /2.
$$
Finally, the Galois conjugation over~$\Q$ shows that 
$$
\ord_{c_1}w_\bfa=\cdots =\ord_{c_{p-1}}w_\bfa,
$$
which implies that ${\ord_{c_k}w_\bfa=-p(p-1)^2 /2}$ for  ${k\in \{1, \ldots, p-1\}}$.
This proves the proposition. \qed \end{proof}

\paragraph{Proof of Theorem~\ref{tspc}}
We have ${S=M_K^\infty}$, the set of the archimedean places of the quadratic field~$K$. 
In particular, ${s=|S|}$ is~$1$ or~$2$. Fix ${P\in Y_\splic(\OO_K)}$. We use the 
notation~$S_i$ and~$\Xi_i$ of Subsection~\ref{sspart};  in the present context this means 
that ${S_2=S\smallsetminus S_1}$ and ${S_3 = M_K^0}$. We again pick for every ${v
\in S_1}$ a $v$-nearby\footnote{Since~$v$ is archimedean, we can write here 
``\textsl{the} $v$-nearby cusp''.} cusp~$c_v$ and set ${\Sigma =\{c_v: v\in S_1\}}$. 
The set~$\Sigma$ has at most two elements, and we have one of the following three 
possibilities:
\begin{align}
\label{einf1}
\Sigma&\subset \{c_\infty, c_1, \ldots, c_{p-1}\},\\
\label{e01}
\Sigma&\subset \{c_0, c_1, \ldots, c_{p-1}\},\\
\label{einf0}
\Sigma&\subset \{c_\infty, c_0\}.
\end{align}
We define ${w=w_{(1/p,0)}^{-1}}$ in the case~(\ref{einf1}), ${w=w_{(0,1/p)}^{-1}}$ 
in the case~(\ref{e01}) and ${w=w_{(1/p,0)}w_{(0,1/p)}}$ in the case~(\ref{einf0}).
With the notation of Subsection~\ref{ssmunk} and Proposition~\ref{pw}, we have 
$$
N=p,   \qquad |G'|\le |G|/2=(p-1)^2 /2, \quad B=
\begin{cases}
1&\text{in the cases~(\ref{einf1}) and~(\ref{e01})},\\
2&\text{in the case~(\ref{einf0})}.
\end{cases}
$$
Now we argue as Subsection~\ref{ssxi1}, with one very significant distinction: instead of 
the estimate ${(\ord_{c_v}w)/e_{c_v}\ge N^{-1}}$ we use the \textsl{identity}
$$
\frac{\ord_{c_v}w}{e_{c_v}}= 
\begin{cases}
(p-1)^2 /2&\text{in the cases~(\ref{einf1}) and~(\ref{e01})},\\
(p-1)^3 /2&\text{in the case~(\ref{einf0})}.
\end{cases}
$$
In the cases~(\ref{einf1}) and~(\ref{e01}) we obtain
\begin{align}
\Xi_1 & \le d^{-1}\sum_{v\in S_1} d_v \frac{\ord_{c_v}w}{e_{c_v}}\log
|q_{c_v}(P)|_v + \frac{1}{2} (p-1)^2d^{-1}\sum_{v\in M_K^\infty}d_v  \rho_v 
\nonumber\\
&=\frac{1}{2} (p-1)^2 d^{-1}\sum_{v\in S_1} d_v \log|q_{c_v}(P)|_v +6(p-1)^2p\log 
p \nonumber\\
&\le -\frac{1}{2}(p-1)^2 d^{-1}\sum_{v\in S_1} d_v \log|j(P)|_v +\frac{1}{2}
(p-1)^2 \log 2+6(p-1)^2 p
\log p \nonumber\\ 
&\le -\frac{1}{2}(p-1)^2 d^{-1}\sum_{v\in M_K} d_v \log^+|j(P)|_v +\frac{1}{2}
(p-1)^2 \log 7000+6(p-1)^2p\log p \nonumber\\ 
\label{exi1}
&= -\frac{1}{2}(p-1)^2 \height(P)+\frac{1}{2}(p-1)^2 \log 7000+6(p-1)^2p\log p. 
\end{align}
In the case~(\ref{einf0}) a similar calculation gives
\begin{equation}
\Xi_1\le -\frac{1}{2}(p-1)^3 \height(P)+\frac{1}{2}(p-1)^3 \log 7000+12(p-1)^2p
\log p.
\end{equation}
We estimate~$\Xi_2$ exactly as in Subsection~\ref{ssxi2}:
\begin{equation}
\label{exi2}
\Xi_2\le \frac{B}{2}(p-1)^2 p\log 5900+6B(p-1)^2p\log p.  
\end{equation}
Further, in the cases~(\ref{einf1}) and~(\ref{e01}) we may take ${\lambda=(1-
\zeta_p)^{6p(p-1)^2}}$, which is equal to $p^{6p(p-1)}$ times a unit. And in the 
case~(\ref{einf0}) we may take ${\lambda=1}$. We obtain
$$
\Xi_3 
\begin{cases}
\le 6(p-1)p\log p&\text{in the cases~(\ref{einf1}) and~(\ref{e01})},\\
=0&\text{in the case~(\ref{einf0})}.
\end{cases}
$$
Combining all the previous estimates, we obtain ${\height(P) \le 24p\log3p}$ in the 
cases~(\ref{einf1}) and~(\ref{e01}), and in the case~(\ref{einf0}) we obtain a much 
sharper estimate ${\height(P) \le 72\log3p}$. This proves the theorem. \qed

\section{An Application: $\Q$-curves of Prime Power Degree}
\addtocontents{toc}{\vspace{-0.7\baselineskip}}
\label{sqcurves}
In this section we apply Theorem~\ref{tspc} to the study of $\Q$-curves, proving 
Theorem~\ref{txpr}. Let us recall some definitions. For a positive integer~$N$ with prime 
decomposition ${N=p_1^{a_1}\cdots p_k^{a_k}}$, we set ${X_0^* (N):=X_{0} (N)/
\langle w_{p_{i}} \rangle}$, where  $\langle w_{p_{i}} \rangle$ is the group of 
automorphisms of $X_{0} (N)$ spanned by the Atkin-Lehner involutions $w_{p_{i}}$, and 
${X_0^+ (N):=X_{0} (N)/w_N }$. If~$K$ is a number field, it is a theorem of 
Elkies~\cite{El04} that any $K$-curve of degree~$N$ gives rise to a point in $X_0^* (N)
(K)$. The curve $X_0^+ (N)(=X_0^* (N)$ if~$N$ is a prime power) parameterizes quadratic 
$K$-curves of degree~$N$ (that is, elliptic curves defined over a quadratic extension of~$K$ 
and admitting a cyclic isogeny of degree~$N$ to the $K$-conjugate  curve). 

Theorem~\ref{txpr} is a consequence of the following three statements.

\begin{theorem}
\label{tmmm1}
For a prime ${p\geq 37}$ and $r>1$, the rational noncuspidal points of $X_0^+ (p^r )$ are 
integral; that is, the $j$-invariant $j(P)$ of any lift $P\in X_0 (p^r )(\C )$ of any 
non-cuspidal point of $X_0^+ (p^r )(\Q )$ belongs to~$\overline{\Z}$.
\end{theorem}

\begin{theorem}
\label{tpel}
For every~$d$ there is a positive number $\kappa(d)$ such that the following holds. 
Let~$E$ be a non-CM elliptic curve  defined over a number field~$K$ of degree~$d$, and 
admitting a cyclic isogeny over~$K$ of degree $\delta$. Then $\delta\le {\kappa(d)
\left(1+\height(j_E)\right)^2}$. 
\end{theorem}

\begin{theorem}
\label{tx0}
Let ${p\ge 3}$ be a prime number,~$K$ be a quadratic number field, 
${r>1}$ an integer, and~$P$ a point of $Y_0 (p^r )(\OO_K )$. Then ${\height (P)\le 
110p\log p}$.
\end{theorem}

\paragraph{Proof of Theorem~\ref{txpr} (assuming 
Theorems~\ref{tmmm1},~\ref{tpel} and~\ref{tx0})} Existence of the degeneracy 
morphisms ${X_{0}^+ (p^{r+2})\to 
X_{0}^+ (p^r )}$ over $\Q$ (see, for instance,~\cite{Mo86}) shows it is enough to prove 
the result for ${r=2}$ and ${r=3}$. The case ${r=2}$, where $X_{0}^+ (p^2 )$ is 
isomorphic to $X_{\mathrm{split}} (p)$,  was settled in~\cite{BP09}, so we are left with 
the case ${r=3}$.

Thus, let ${P\in X^+_{0} (p^3)(\Q )}$ be a non-cuspidal and non-CM point, and ${Q\in 
X_{0} (p^3)(K)}$ a lift of it, with values in a quadratic number field~$K$. From 
Theorem~\ref{tmmm1} we know that~$Q$ is an integral point if~$p$ is large enough. 
Theorem~\ref{tx0} implies that ${\height(Q)\le 110p\log p}$. 

Call~$E$ the elliptic curve associated to~$Q$.  It is a non-CM elliptic curve admitting a 
cyclic isogeny of degree~$p^3$ over~$K$. Theorem~\ref{tpel} implies that ${p^3\le C
\bigl(1+\height(Q)\bigr)^2}$ with an absolute constant~$C$. Therefore ${p^3\le C' (p
\log p)^2}$ with another absolute constant~$C'$ and~$p$ is bounded.\qed

\bigskip

Theorem~\ref{tmmm1} will be proved in Section~\ref{sapp}. It is a generalization of the 
one used in \cite{BP09} when ${r=2}$, which in that case was originally due to Mazur, 
Momose and Merel. Theorem~\ref{tpel} is a straightforward consequence of  the isogeny 
bounds due to   Masser and W\"ustholz~\cite{MW90} and Pellarin~\cite{Pe01}. See 
\cite[Corollary 5.4]{BP09} for the details.

Theorem~\ref{tx0} is deduced below from Theorem~\ref{tspc} and 
the following lemma.

\begin{lemma}
\label{lfalt}
Let~$E$,~$E'$ be elliptic curves defined over some number field and linked by an isogeny of 
degree~$\delta$. Then
$$
\bigl|\height(j_E)-\height(j_{E'})\bigr|\le 13\log\bigl(1+\height(j_{E'})\bigr)+
7\log\delta+ 100. 
$$
\end{lemma}

\begin{proof}
Denote by $\height_{\calF}(E)$ the \textsl{Faltings semistable height} of the elliptic 
curve~$E$. Recall that $h_{\calF}(E)$ is defined as ${[K :\Q ]^{-1} \deg \omega }$, 
where~$K$ is a  number field such that~$E$ has semi-stable reduction at every place of~$K$,
and~$\omega$ is a N\'eron differential on~$E\vert_{K}$; it is independent of the choice 
of~$K$ and~$\omega$. A result of Faltings \cite[Lemma 5]{Fa83} implies that 
$$
\bigl|\height_{\calF} (E) -
\height_{\calF} (E')\bigr| \le \frac{1}{2}\log \delta.
$$
Further, for any elliptic curve~$E$ over  a number field we have
$$
\bigl|\height (j_{E})-12\height_{\calF} (E )\bigr|\leq 6\log \bigl(1+\height 
(j_{E})\bigr)+C,
$$
with an absolute constant~$C$, see \cite[Proposition~2.1]{Si84}. Pellarin shows that one 
take ${C=47.15}$,  see~\cite{Pe01}, equation~(51) on page~240. Combining all this, we 
find
$$
\bigl|\height(j_E)-\height(j_{E'})\bigr|\le 6\log\bigl(1+\height(j_E)\bigr)+6
\log\bigl(1+\height(j_{E'})\bigr)+6\log\delta+ 95,
$$
which implies the result after a routine calculation. \qed
\end{proof}

\paragraph{Proof of Theorem~\ref{tx0}}
We may assume ${r=2}$. Let ${\phi:X_0(p^2)\to X_\splic(p)}$ be the isomorphism 
defined in the beginning of Section~\ref{sspc}. Then the elliptic curve implied by a
point~$P$ on $X_0(p^2)$ is $p$-isogenous to the curve implied by the point ${P'=\phi(P)}$ 
on $X_\splic(p)$. Lemma~\ref{lfalt} implies that 
\begin{equation}
\label{efaltp}
\bigl|\height(P)-\height(P')\bigr|\le 13\log\bigl(1+\height(P')\bigr)+7\log p
+100. 
\end{equation}
Since ${P\in Y_0(p^2)(\OO_K)}$ and good reduction is preserved under isogeny,~$P'$ 
belongs to~$Y_\splic(p)(\OO_K)$ as well. Applying Theorem~\ref{tspc} 
to~$P'$, we find ${\height(P')\le 24p\log3p}$, which, combined with~(\ref{efaltp}), 
implies the result.\qed

\section{Integrality of $Y_0^+ (p^3 )(\Q )$}
\addtocontents{toc}{\vspace{-0.7\baselineskip}}
\label{sapp}
We show that rational points on $X_0^+ (p^3 )$ are, in fact, integral.  

\begin{theorem}
\label{tmmm}
For a prime ${p\geq 37}$, and $P\in X_0^+ (p^3 )(\Q )$ a non-cuspidal
point, the $j$-invariant of any lift of $P$ to $X_0 (p^3 )(\C )$ belongs to~$\bar\Z$.
\end{theorem}
The proof of this theorem is an adaptation of the one we proposed in~\cite{BP09}, relying
on results and observations of Mazur, Momose and Merel. Actually, Theorem 8.1 was already
proven, except integrality at $2$ when $p\not\equiv 1\mod 8$, by Momose 
in~\cite[Theorem 3.8]{Mo86}; in the present paper we however do neeed the whole
statement of Theorem~\ref{tmmm}. The theorem is probably true for $p\geq 11$, $p\neq 
13$ (some cases are indeed given by Momose in loc. cit.); but our assumption that $p\geq 
37$ simplifies our arguments.

  If $M, N$ are natural integers and $M$ is a divisor of $N$, we write 
$\pi_{N,M} \colon X_0 (N)\to X_0 (M)$ for the degeneracy morphism which is defined 
functorially as $(E,A_N )\mapsto (E,A_M )$, where $A_M := E[M]\cap A_N$. If~$M$ and 
$N/M$ are relatively prime, let $w_M$ for the corresponding Atkin-Lehner involution; recall 
that~${X_0^+ (N) :=X_0 (N)/w_N}$. As usual, we write $J_0 (N)$ for the jacobian 
over~$\Q$ of $X_0 (N)$, and ${J_0^- (N):=J_0 (N)/(1+w_N )J_0 (N)}$. Models over rings 
of integers for abelian varieties will be N\'eron models. Recall that, in this paper, the model 
for~$X_0 (N)$ over $\Z$ that we consider is the modular one. Models for 
those modular curves over arbitrary schemes will be deduced by base change. We denote by 
$X_0 (N)^{\mathrm{sm}}_{\Z}$ the smooth locus of $X_0 (N)_{\Z}$ (obtained, when 
$N=p$ is prime, by removing the (super)singular points in the fiber at $p$). 

As already mentioned in Section~\ref{sspc}, the curve $X_\splic(N)$ 
parametrizes elliptic curves endowed with an ordered pair of independent $N$-isogenies.  
With ``ordered'' replaced by ``unordered'', the same is true (at least when~$N$ is a prime 
power) for the curve $X_{\mathrm{split}} (N)$. For each prime~$p$ dividing~$N$ there is 
an involution on $X_\splic(N)$, here also denoted by $w_{p}$,  defined functorially by 
$$
\left(E,\Bigl(A=\prod_{q}A_{q}
,B=\prod_{q} B_{q}\Bigr)\right)\mapsto \left(E,\Bigl(\prod_{q\neq p}A_q \times
B_{p} ,\prod_{q\neq p}B_q \times A_{p}\Bigr)\right),
$$ 
so that ${X_{\mathrm{split}} (N)=X_\splic (N)/\langle w_{p}:{p|N}\rangle}$. The 
map ${z\mapsto Nz}$ on the upper half-plane defines the $\Q$-isomorphism ${\phi\colon 
X_0 (N^2 )\simeq X_{\mathrm{sp. C.}} (N)}$ of Subsection~\ref{sspc}, inducing an 
isomorphism 
$$
X_{\mathrm{split}} (N)
\simeq X_0 (N^2 )/\langle w_{p}:{p|N}\rangle
$$ 
on the quotients.

 We recall certain properties of the modular Jacobian $J_0(p)$ and its 
\textsl{Eisenstein quotient} $\tilJ(p)$ (see~\cite{Ma77}). 

\begin{sloppypar}

\begin{proposition}
\label{pmaz}
Let~$p$ be a prime number. Then we have the following.

\begin{enumerate}

\item \cite[Theorem~1]{Ma77}\quad 
The group $J_0(p)(\Q)_\tors$ is cyclic and generated by $\cl(0-\infty)$, where~$0$ 
and~$\infty$ are the cusps of $X_0(p)$. Its order is equal to the numerator of the quotient
${(p-1)/12}$.

\item \cite[Theorem~4]{Ma77}\quad 
The group $\tilde J(p)(\Q)$ is finite. Moreover, the natural projection ${J_0(p)\to \tilde 
J(p)}$ defines an isomorphism ${J_0(p)(\Q)_\tors \to \tilJ(p)(\Q)}$.

\end{enumerate}

\end{proposition} 
\begin{remark}
As Mazur notices, Raynaud's theorem on group schemes of type $(p,\ldots ,p)$ insures that
$J_0(p)(\Q)_\tors$ defines a $\Z$-group scheme which, being constant in the generic fiber, 
is constant outside $2$, and which at 2 has \'etale quotient of rank at least half that of $J_0(p)
(\Q)_\tors$.
\end{remark}
\end{sloppypar}

  For a point ${Q\in X_0^+ (p^3 )(\Q )}$ and ${z\in X_0 (p^3 )(K )}$ a lifting of $Q$ with 
$K$ a quadratic number field, the point $z$ corresponds to a couple $(E,C_{p^3})$ over $K$,
by \cite{DR73}, Proposition VI.3.2. Set $\pi :=\pi_{p^3 ,p}$, $x:=w_p \pi (z)$ and 
$x_0 :=\pi w_{p^3} (z) \in X_0 (p)(K)$. Writing $D_p :=p^2 C_{p^3}$, the modular 
interpretation of $x$ and $x_0$ is therefore $(E/D_p ,E[p] \mod D_p )$ and $(E/C_{p^3} ,
E[p]+C_{p^3} \mod C_{p^3} )$ respectively. For $t$ an element in the $\Z$-Hecke algebra 
for $\Gamma_{0} (p)$, define the morphism~$g_t$ from $X_0 (p^3 )^{\mathrm{sm}}_{
/\Z}$ to $J_{0} (p)_{/\Z}$, which extends by the universal property of N\'eron models 
the morphism on generic fibers:
$$
g_{t} \colon
\left\{
\begin{array}{rcl}
X_0 (p^3 ) & \to & J_{0} (p) \\
Q & \mapsto & t\cdot {\mathrm{cl}} \bigl( w_p \pi (Q) -\pi w_{p^3} (Q) \bigr) .
\end{array}
\right.
$$
Let ${J_{0} (p)\stackrel\Pi\to \tilJ(p)}$ be the projection to the Eisenstein quotient, 
and ${\tilg_t :=\Pi\circ g_t}$. 
\begin{lemma} The morphism $\tilg_{t}$ factorizes through a $\Q$-morphism 
$\tilg_{t}^+$ from $X_{0}^+ (p^3 )$ to $\tilJ (p)$. If $t\cdot (1+w_{p})=0$, the same 
is true for $g_{t}$ and we similarly denote by $g^+_{t} \colon X_{0}^+ (p^3 ) \to J_{0} 
(p)$ the factor morphism.
\end{lemma}
\begin{proof} We compute that, in $J_{0} (p)$: 
$${\mathrm{cl}}((w_p \pi (z))-(\pi w_{p^3} (z))-(w_p \pi w_{p^3} (z))+(\pi (z))) =(1+
w_p ){\mathrm{cl}}((\pi (z))-(\pi w_{p^3} (z))),
$$
from which we derive the second assertion when $t(1+w_{p})=0$. As for the first statement, 
we know that $(1+w_{p})$ acts trivially on $\tilJ(p)$ 
from~\cite[Proposition 17.10]{Ma77}. 
\qed
\end{proof}
\bigskip
By the universal property of N\'eron models, we may extend $g_{t}^+$ and $\tilg_{t}^+$ 
to maps from $X^+_0 (p^3)^{\mathrm{sm}}_{/\Z}$ to  $J_{0}(p)_{/\Z}$ and 
$\tilJ(p)_{/\Z}$, respectively. We still denote those extended morphisms in the same way 
and we put ${\tilg^+=\tilg^+_1}$.  
\begin{proposition}
\label{peab}
Let ${P\in X_0^+ (p^3 )(\Q)}$ for some $p\geq 37$, let $x$, $x_{0}$ and~$K$ be defined 
as above, and let ${\cal O}_K$ be the ring of integers of $K$. Then:
\begin{enumerate}
\item
\label{insup}
The isogeny class of elliptic curves associated to~$P$ is not potentially supersingular 
at~$p$. 
\item
\label{ipfiber}
The points $x$ and $x_{0}$ coincide in the fibers of characteristic $p$ of $X_0
(p)_{/{\cal O}_K}$. 
\end{enumerate}
\end{proposition}

\begin{proof}
Point~(\ref{insup}) is Lemma 2.2 (ii) together with Theorem 3.2 
of~\cite{Mo86}. Point~(\ref{ipfiber}) is proved in Proposition 3.1 of~\cite{Pa05}. 
\qed
\end{proof}

\paragraph{Proof of Theorem~\ref{tmmm}}
Let~$P$ be a non-cuspidal point on $X_0^+ (p^3 )$ with values in $\Q$, and $Q$ a lift of 
$P$ to $X_{0} (p^3 )(\Q )$. If $\cal L$ is a finite place of ${\cal O}_{K}$ dividing the 
denominator of $j(Q)$, then~$Q$ specializes to a cusp at~$\cal L$. Recall that $X_0 
(p^3 )$ has two cusps defined over $\Q$, and two other Galois orbits of cusps, with
fields of definition $\Q (\zeta_{p})^+$ and $\Q (\zeta_{p^2})^+$ respectively. We first 
claim that~$Q$ specializes to one of the rational cusps (which, by changing our lift of
$P$, may be assumed to be the $\infty$-cusp, as $w_{p}$ switches the rational cusps). 
Indeed, it follows from 
Propositions~\ref{peab}~(\ref{ipfiber}) that ${\tilg^+ (P)(\F_p )=0(\F_p )}$, and by 
the remark after Proposition~\ref{pmaz}, ${\tilg^+ (P)(\Q )=0(\Q )}$ (recall $p>2$). The 
non-rational cusps of $X_0^+ (p^3 )(\C )$ map to $\mathrm{cl} (0-\infty)$ in $J_{0} (p)
(\C )$ (indeed, $w_{p^3}$ preserves each non-rational Galois orbit of cusps. For more 
details see for instance the proof of~\cite[Proposition 2.5]{Mo84}). 
Therefore, as $p\geq 37$, Proposition~\ref{pmaz} implies that if $Q$ 
specializes to a non-rational cusp at $\cal L$ then $\tilg ^+(P)$ would not be $0$ at the
characteristics $\ell$ of $\cal L$, a contradiction. 

Now let~$\calI$ be the ideal of the Hecke algebra such that~${\tilde J (p)}=J_0(p)/\calI 
J_{0} (p)$. Choose an $\ell$-adically maximal element~$t\neq 0$ in the Hecke 
algebra such that~$t\cdot \calI =0$. Again, as $t(1+w_{p}) =0$, the morphism $g_{t}$ 
factorizes through a morphism $g_{t}^+ \colon X_0^+ (p^3 )^{\mathrm{sm}}_{/\Z}\to 
t\cdot J_0  
(p)_{/\Z}$. Moreover $g_{t}^+ (P)$ belongs to $t\cdot J_{0} (p)(\Q )$, hence is a torsion 
point, as $t\cdot J_{0} (p)$ is isogenous to a quotient of $\tilde J 
(p)$. We see as above by looking at the fiber at $p$ that $g_{t}^+ (P)=0$ at $p$, hence 
generically, because of Proposition~\ref{pmaz} (or more generally by Raynaud's 
well-known result on group schemes of type $(p,\dots ,p)$ on a not-too-ramified discrete 
valuation ring). We then easily check by using the $q$-expansion principle, as 
in~\cite[Theorem 5]{Me05}, that $g_{t}^+$ is a formal immersion at the specialization 
$\infty (\F_\ell )$ of the rational cusp~$\infty$ on $X_0^+ (p^3 )$. This allows us to 
apply the classical argument of Mazur (see e.g. \cite[proof of Corollary 4.3]{Ma78}), 
yielding a contradiction with the fact that $P$ is not generically cuspidal. Therefore $P$ is
not cuspidal at $\cal L$.  \qed

\end{document}